\documentclass[12pt, 14paper,reqno]{amsart}
\usepackage{amsmath,amsfonts,amssymb}
\usepackage{mathrsfs}
\usepackage{longtable}
\usepackage[breaklinks]{hyperref}
\usepackage{graphicx}

\theoremstyle{plain}
\newtheorem{thm}{Theorem}[section]
\newtheorem{lem}{Lemma}[section]
\newtheorem{cor}{Corollary}[section]

\newtheorem{thma}{Theorem}

\theoremstyle{proof}

\numberwithin{equation}{section}

\begin{document} 
\title[On the Diophantine equation $dx^2+p^{2a}q^{2b}=4y^p$]{On the Diophantine equation $dx^2+p^{2a}q^{2b}=4y^p$}
\author{Kalyan Chakraborty and Azizul Hoque}
\address{Kalyan Chakraborty @Kerala School of Mathematics, Kozhikode-673571, Kerala, India.}
\email{kalychak@ksom.res.in}
\address{Azizul Hoque @Department of Mathematics, Rangapara College, Rangapara, Sonitpur-784505, India.}
\email{ahoque.ms@gmail.com}
\date{\today}
\maketitle

\begin{abstract}
We investigate the solvability of the Diophantine equation in the title, where $d>1$ is a square-free integer, $p, q$ are distinct odd primes and $x,y,a,b$ are unknown positive integers with $\gcd(x,y)=1$. We describe all the integer solutions of this equation, and then use the main finding to deduce some results concerning the integers solutions of some of its variants. The methods adopted here are elementary in nature and are primarily based on the existence of the primitive divisors of certain Lehmer numbers.
\end{abstract}

\section{Introduction} Many special cases of the Diophantine equation,
\begin{equation}\label{eqi1}
dx^2+c^m=4y^n,~~ x,y, m, n\in \mathbb{N}, n\geq 3,
\end{equation}
where $c$ and $d$ are given positive integers, have been considered by several authors over the decades. In particular, there are many interesting results about the integer solutions of this equation for $d=1$ and we direct the reader to the papers \cite{BHS19, CHS20, LE92, LJ72, LTT09} for more information. For a survey on this very interesting subject we recommend \cite{LS20, SH20}. On the other hand \eqref{eqi1} for square-free $d$ and odd $m$ was well investigated under the conditions $\gcd(x,y)=1$ and $\gcd(n, 2h(-cd))=1$, where $h(\Delta)$ denotes the class number of $\mathbb{Q}(\sqrt{\Delta})$. The main results in this case include the following:
\begin{itemize}
\item[$\square$] Le \cite{LE-95} proved that for prime $n>5$, \eqref{eqi1} has only finite number of integer solutions $(x,y,m,n)$.  Moreover, the solutions satisfy $4y^n<\exp\exp 470$. Later, this bound was improved by Mignotte \cite{MI97}. 

\item[$\square$] Bugeaud \cite{BU01}, and Arif and Al-Ali \cite{AA02} independently, determined the solutions of \eqref{eqi1} when prime $n\geq 5$. However, Bilu \cite{BI02} pointed out a flaw in \cite{LE-95, MI97} and {\it a fortiori} Bugeaud's result in \cite{BU01}, and corrected that inaccuracy. 
\end{itemize}
Recently, Dabrowski et al. \cite{DGS20} investigated \eqref{eqi1} for square-free $d$ and even $m$ under the conditions $\gcd(x,y)=1$ and $\gcd(n, 2h(-d)) = 1$. 
More precisely, they completely solved \eqref{eqi1} for $d \in \{7, 11, 19, 43, 67, 163\}$, and $c$ a power of an odd prime, under the conditions $2^{n-1}c^{m/2}\not\equiv\pm 1\pmod d$, $\gcd(n, c) = 1$ and $n$ odd prime. The authors extended these results in \cite{CHS21} for any square-free $d$ by removing the above assumptions under the condition that $c$ is an odd prime. 

In this paper,  we consider the Diophantine equation 
\begin{equation}\label{eqn1}
dx^2+p^{2m}q^{2n}=4y^p,~~ x,y,m,n\in\mathbb{N},  \gcd(x,y)=1, \gcd(p, h(-d))=1,   
\end{equation}
where $d$ is a square-free positive integers, and $p$ and $q$ are distinct odd primes. This equation has been studied in \cite{CHS21} when the case $m=0$. The methods adopted here are elementary in nature and are based on the primitive divisors of certain Lehmer numbers. 

To proceed further, we fix up the following notations for a given positive odd integer $k$:
$$\mathcal{R}(d,u,v,k):=\sum_{j=0}^{\frac{k-1}{2}} \binom{k}{2j}u^{k-2j-1}d^{\frac{k-1}{2}-j}(-v^2)^j,$$
$$\mathcal{I}(d,u,v,k):=\sum_{j=0}^{\frac{k-1}{2}} \binom{k}{2j+1}u^{k-2j-1}d^{\frac{k-1}{2}-j}(-v^2)^j.$$
Then 
\begin{equation}\label{eqr}
\mathcal{R}(d,u,v,k)\equiv \begin{cases} u^{k-1}d^{(k-1)/2} \pmod k,\\
(-1)^{(k-1)/2}k v^{k-1} \pmod d,\\
u^{k-1}d^{(k-1)/2}\pmod {v^2}.
\end{cases}
\end{equation}
and 
\begin{equation}\label{eqi}
\mathcal{I}(d,u,v,k)\equiv \begin{cases} (-1)^{(k-1)/2} v^{k-1} \pmod k,\\
(-1)^{(k-1)/2} v^{k-1} \pmod d,\\
ku^{k-1}d^{(k-1)/2}\pmod{v^2}.
\end{cases}
\end{equation}

Here, our main objective is to prove the following result and then we will discuss some of its consequences.
\begin{thm}\label{thm}
Let $d\geq 1$ be a square-free integer and let $p$ and $q$ be distinct odd primes satisfying $p\nmid h(-d)$. Then the following hold:
\begin{itemize}
\item[(I)] If $d\equiv 1,2\pmod 4$, then \eqref{eqn1} has no integer solution.
\item[(II)] If $d\equiv 3\pmod 4$, then  \eqref{eqn1} has no positive integer solution when $q^n\not\equiv \pm 1\pmod p$. Furthermore when $q^n\equiv \pm 1\pmod p$, the possible positive integer solutions are given by 
$$(x,y)=\left(\frac{|u\mathcal{R}(d,u,p^{m-1},p)|}{2^{p-1}},  \frac{u^2d+p^{2(m-1)}}{4}\right),$$
for some suitable positive integer $u$. 
\end{itemize}
\end{thm}

\subsection*{Remarks} We note down some comments on Theorem \ref{thm}. 
\begin{itemize}
\item[(i)] For $d=1$ or $d=\square$, it is not difficult to see that \eqref{eqn1} has no solutions by reading it modulo $4$.

\item[(ii)] The condition `$q^n\equiv \pm 1\pmod p$' can be replaced by a bit stronger condition `$q^n\not\equiv \pm y^{(p-1)/2}\pmod{p^{2(m-1)}}$' by utilizing \eqref{eqp3} and \eqref{eqf1}.

\item[(iii)] From \eqref{eqr}, we have $\mathcal{R}(d,u,p^{2(m-1)}, p)\equiv d^{(p-1)/2}\pmod p$, which implies that $\mathcal{R}(d,u,p^{2(m-1)}, p)\equiv \pm 1 \pmod p$. Therefore $x\equiv \pm u\pmod p$.
\end{itemize}

Let $p\geq 5$ be a prime such that $p+2$ is also a prime. Then $(p+2)^p\equiv 2\pmod p$, and  thus Theorem \ref{thm} yields:
\begin{cor}\label{cor1}
Let $d$ be a square-free positive integer and let $p\geq 5$ be a prime such that $p+2$ is also a prime and $p\nmid h(-d)$.  Then the Diophantine equation
$$dx^2+p^{2m}(p+2)^{2p}=4y^p,~~ x,y,m\in \mathbb{N}, \gcd(x,y)=1,$$
has no integer solution. 
\end{cor}
Let $d\in\{2,3,7,11,19,43,67,163\}$ and let $p>41$ be a prime. Then $(d+p)^p\equiv d\pmod p$ and we have:
\begin{cor}\label{cor1}
Let $p>41$ be a prime and let $d\in\{2,3,7,11,19,43,67,163\}$ such that $d+p$ is a prime. Then the Diophantine equation
$$dx^2+p^{2m}(d+p)^{2p}=4y^p,~~ x,y,m\in \mathbb{N}, \gcd(x,y)=1,$$
has no integer solution. 
\end{cor}

Let $\mathfrak{A}=\{7, 11,15,19,35,39,43,51,55,67,91,95,111,115,123,155,163,183, 
187,195,203,\\219,235,259,267, 295,299,323,355,371,~395,~399,403,~407,427,435,471,
483, 555,559,579,\\583,595,627,651,663,667,715,723,763,791,795,799,895,903,915,939,
943,~955,~979,~987,\\ 995,1003,1015,1023,1027,1043,1047,1119,1131,1139,1155,1159,~1195,
~1227,~1239,~1243,\\1299,1339,1379,1387, 1411,1435,1443,1463,1507,1551,1555,1595,1635,
1651,1659,1731
\}$. Then for $d\in \mathfrak{A}$, we have $h(-d)\in\{1,2,4,8,16,32\}$. 
Thus the following corollary follows from Theorem \ref{thm}.
\begin{cor}
Let $d\in\mathfrak{A}$ and let $p\geq 5$ be an odd prime. 
Then the equation
$$dx^2+3^{2p}p^{2m}=4y^p,~~ x, y\geq 1, m\geq 1, \gcd(x,y)=1,$$ 
has no integer solutions.
\end{cor}

\section{Proof of Theorem \ref{thm}}
We begin with the following lemma which is extracted from  \cite[Corollary 3.1]{YU05}.
\begin{lem}\label{lemYU}
Let $d$ be a square-free positive integer. Assume that $k$ is a positive odd integer such that $\gcd(k, h(-d))=1$. Then all the positive integer solutions $(x,y,z)$ of the equation 
\begin{equation*}\label{eqyu}
dx^2+y^2=4z^k,~~ \gcd(dx,y)=1,
\end{equation*} 
can be expressed as 
$$\frac{x\sqrt{d}+y\sqrt{-1}}{2}=\lambda_1\left(\frac{u\sqrt{d}+\lambda_2 v\sqrt{-1}}{2}\right)^k,$$
where $u$ and $v$ are positive integers satisfying $4z=du^2+v^2$, $\gcd(du, v)=1$, and $\lambda_1, \lambda_2\in\{-1, 1\}$ when $d>3$. For $d=3$, $\lambda_1\in\{\pm 1, (\pm 1\pm \sqrt{-d})/2 \}$ and $\lambda_2=\pm 1$.
\end{lem}

Given a fixed square-free positive integer $d$, let $(p,q,x,y,m, n)$ be a positive integer solution of \eqref{eqn1}. Then $x$ is odd since both $p$ and $q$ are odd, and hence \eqref{eqn1} modulo $4$ gives $d\equiv 3\pmod 4$. This confirms that \eqref{eqn1} has no positive integer solution when $d\equiv 1,2\pmod 4$. 

Utilizing the facts $\gcd(x,y)=1$ and $d$ is square-free in \eqref{eqn1}, we get $\gcd(pq, dx)=1$. Since $p$ is odd prime satisfying $p\nmid h(-d)$, so that by Lemma \ref{lemYU} we can find positive integers $u$ and $v$ such that 
\begin{equation}\label{eqp2}
\frac{x\sqrt{d}+p^mq^n\sqrt{-1}}{2}=\lambda_1\left(\frac{u\sqrt{d}+\lambda_2 v\sqrt{-1}}{2}\right)^p
\end{equation}
 and 
 \begin{equation}\label{eqp3}
 u^2d+v^2=4y
 \end{equation}
 with $\gcd(ud,v)=1$, where $\lambda_1, \lambda_2\in\{-1, 1\}$. Note that for $d=3$, we have $\lambda_1=\pm 1, (\pm1\pm \sqrt{-3})/2$ which satisfy $\lambda_1^6=1$ and thus it can be absorbed into the $p$-th power except for $p=3$.  The case $(d, p)=(3,3)$ is not possible by the hypothesis.

We now equate the real and imaginary parts in \eqref{eqp2} to get
\begin{align}\label{eqp4}
\begin{cases}
x=\dfrac{\lambda_1 u}{2^{p-1}}\mathcal{R}(d,u,v,p),\vspace{2mm}\\
p^mq^n=\dfrac{\lambda_1 \lambda_2 v}{2^{p-1}}\mathcal{I}(d,u,v,p).
\end{cases}
\end{align}
Since both $x$ and $pq$ are odd, \eqref{eqp4} shows that both $u$ and $v$ are odd.

Let us define,
\begin{equation}\label{eqp5}
\begin{cases}
\alpha=\dfrac{u\sqrt{d}+\lambda_2 v \sqrt{-1}}{2},\\
\beta=\dfrac{u\sqrt{d}-\lambda_2 v\sqrt{-1}}{2}.
\end{cases}
\end{equation}
\subsection{Lehmer numbers and existence of their primitive divisors}
A pair $(\alpha, \beta)$ of algebraic integers is said to be a Lehmer pair if $(\alpha + \beta)^2$ and $\alpha\beta$ are two non-zero coprime rational integers, and $\alpha/\beta$ is not a root of unity. For a given positive integer $n$, the $n$-th Lehmer number corresponds to the pair $(\alpha, \beta)$ is defined as 
$$\mathcal{L}_n(\alpha, \beta)=\begin{cases}
\dfrac{\alpha^n-\beta^n}{\alpha-\beta} & \text{ if $n$ is odd,} \vspace{1mm}\\
\dfrac{\alpha^n-\beta^n}{\alpha^2-\beta^2} & \text{ if $n$ is even}.
\end{cases}$$
Note that all Lehmer numbers are non-zero rational integers.  The Lehmer pairs $(\alpha_1, \beta_1)$ and $(\alpha_2, \beta_2)$ are said to be equivalent if $\alpha_1/\alpha_2=\beta_1/\beta_2\in \{\pm 1, \pm\sqrt{-1} \}$. A prime divisor $p$ of $\mathcal{L}_n(\alpha, \beta)$ is said to be primitive if $p\nmid(\alpha^2-\beta^2)^2
\mathcal{L}_1(\alpha, \beta) \mathcal{L}_2(\alpha, \beta) \cdots \mathcal{L}_{n-1}(\alpha, \beta)$. 
 The following classical result concerning the existence of primitive divisors of Lehmer numbers was proved in \cite[Theorem 1.4]{BH01}. 
\begin{thma}\label{thmBH}
For any integer $n>30$, the Lehmer numbers $\mathcal{L}_n (\alpha, \beta) $ have primitive divisors.
\end{thma}

Given a Lehmer pair $(\alpha, \beta)$, let $a=(\alpha+\beta)^2$ and $b=(\alpha-\beta)^2$. Then $\alpha=(\sqrt{a}\pm\sqrt{b})/2$ and $\beta=(\sqrt{a}\mp\sqrt{b})/2$. This pair $(a, b)$ is called the parameters corresponding to the Lehmer pair $(\alpha, \beta)$. The following lemma is extracted from \cite[Theorem 1]{VO95}.
\begin{lem}\label{lemVO}
Let $p$ be a prime such that $7\leq p\leq 29$. If the Lehmer numbers $\mathcal{L}_p(\alpha, \beta)$ have no primitive divisor, then up to equivalence, the parameters $(a, b)$ of the corresponding pair $(\alpha, \beta)$ are as follows:
\begin{itemize}
\item[(i)] $(a, b)=(1,-7), (1, -19), (3, -5), (5, -7), (13, -3), (14, -22)$ when $p=7$;
\item[(ii)] $ (a, b)=(1,-7)$ when $p=13$.
\end{itemize}
\end{lem}
Let $F_k$ (resp. $L_k$) denote the $k$-th term in the Fibonacci (resp. Lucas) sequence defined by $F_0=0,   F_1= 1$,
and $F_{k+2}=F_k+F_{k+1}$ (resp. $L_0=2,  L_1=1$, and $L_{k+2}=L_k+L_{k+1}$), where $k\geq 0$ is an integer. 
The following lemma is a part of \cite[Theorem 1.3]{BH01}. 
\begin{lem}\label{lemBH}
For $p=3, 5$, let the Lehmer numbers $\mathcal{L}_p(\alpha, \beta)$ have no primitive divisor. Then up to equivalence, the parameters $(a, b)$ of the corresponding pair $(\alpha, \beta)$ are:
\begin{itemize}
\item[(i)] For $p=3, (a, b)=\begin{cases}(1+t, 1-3t) \text{ with } t\ne 1,\\
 (3^k+t, 3^k-3t) \text{ with } t\not\equiv 0\pmod 3, (k,t)\ne(1,1);\\ \end{cases}$ 
 \vspace{2mm}
\item[(ii)] For $p=5, (a, b)=\begin{cases}
(F_{k-2\varepsilon}, F_{k-2\varepsilon}-4F_k)\text{ with } k\geq 3,\\
 (L_{k-2\varepsilon}, L_{k-2\varepsilon}-4L_k)\text{ with } k\ne 1;
 \end{cases}
 $
\end{itemize}
where $t\ne 0$ and $k\geq 0$ are any integers and $\varepsilon=\pm 1$.
\end{lem}
The following few small results will be useful in the sequel.
\begin{lem}\label{lemPD}
The pair $(\alpha, \beta)$ as defined in \eqref{eqp5}  is a Lehmer pair, and the corresponding parameter is $(u^2d, -v^2)$. 
\end{lem}
\begin{proof}
Utilizing \eqref{eqp3}, we see that $\alpha$ satisfies the polynomial $X^4-(2y-v^2)X^2+y^2\in \mathbb{Z}[X]$. Thus $\alpha$ is an algebraic integer and so is $\beta$. Since $\gcd(ud,v)=1$, so that \eqref{eqp3} gives $\gcd(ud, y)=1$, and hence $(\alpha+\beta)^2=u^2d$ and $\alpha\beta=y$ are coprime. 

Now it follows from the following identity $$\frac{u^2d}{y}=\frac{(\alpha+\beta)^2}{\alpha\beta}=\frac{\alpha}{\beta}+\frac{\beta}{\alpha}+2$$ 
that $$y\left(\frac{\alpha}{\beta}\right)^2+(2y-u^2d)\frac{\alpha}{\beta}+y=0.$$
As $\gcd(2y-u^2d, y)=\gcd(ud,y)=1$, so that $\dfrac{\alpha}{\beta}$ is not an algebraic integer and hence it is not a root of unity. This completes the proof. 
\end{proof}

\begin{lem}\label{lemPE} 
Let $p$, $\alpha$ and $\beta$ be as in \eqref{eqp2} and \eqref{eqp5}. Then $q$ is the only possible primitive divisor of $\mathcal{L}_p(\alpha, \beta)$. 
\end{lem}
\begin{proof}
By Lemma \ref{lemPD}, $(\alpha, \beta)$ is a Lehmer pair and thus it determines the Lehmer numbers $\mathcal{L}_t(\alpha, \beta)$ for any integer $t\geq 1$. Since $p$ is odd prime, so that $$\mathcal{L}_p(\alpha, \beta)=\frac{\alpha^p-\beta^p}{\alpha-\beta}.$$
Utilizing \eqref{eqp2} and \eqref{eqp5}, we get 
\begin{equation}\label{eqp6}
|\mathcal{L}_p(\alpha, \beta)|=\left| \frac{p^mq^n}{v}\right|.
\end{equation}
This shows that $v\mid p^mq^n$ as $\mathcal{L}_p(\alpha, \beta)\in \mathbb{Z}$, and thus $p$ and  $q$ are only candidate for the primitive divisors of $\mathcal{L}_p(\alpha, \beta)$. However it is well  known that if $q$ is a primitive divisor of $\mathcal{L}_t(\alpha, \beta)$, then $q\equiv \pm 1 \pmod t$.  This fact confirms that $q$ is the only possible primitive divisor of $ \mathcal{L}_p(\alpha, \beta)$. 
\end{proof}
\subsection{Some properties of Lucas and Fibonacci sequences}
In order to complete the proof of Theorem \ref{thm}, we need these basic results. 
\begin{thma}[{\cite[Theorems 1 and 3]{CO64}}]\label{thmCO} For an integer $k\geq 0$, let $F_k$ (resp. $L_k$) denote the $k$-th Fibonacci (resp. Lucas) number. Then
\begin{itemize}
\item[(i)] if $L_k=x^2$, then $(k,x)=(1,1),(3,2)$;
\item[(ii)] if $F_k=x^2$, then $(k,x)=(0,0), (1,1),(2,1), (12,12)$.
\end{itemize}
\end{thma}

We extract the following lemma from \cite[Theorem 3.3]{KK15}.
\begin{lem}\label{lemKK}
Let $F_k$ be as in Theorem \ref{thmCO}. If $F_k=5x^2$, then $(k, x)=(5,1)$. 
\end{lem}

We also recall the following lemma from \cite[Lemma 2.1]{HO20}.
\begin{lem}\label{lemfl} Let $F_k$ and $L_k$ be as in Theorem \ref{thmCO}.
 Then for $\varepsilon=\pm 1$,
\begin{itemize}
\item[(i)] $4F_k-F_{k-2\varepsilon}=L_{k+\varepsilon}$,
\item[(ii)] $4L_k-L_{k-2\varepsilon}=5F_{k+\varepsilon}$
\end{itemize}
\end{lem}

\subsection{Completion of the proof of Theorem \ref{thm}}
From \eqref{eqp6}, the possible values of $v$ are $p^mq^n, p^{m_1}q^n$ and $p^{m_1}q^{n_1}$ for $0\leq m_1\leq m,~~~ 0\leq n_1\leq n-1$. We will handle these possibilities individually.
\vspace{2mm}\\
$\blacksquare$ When $v=p^mq^n$. In this case \eqref{eqp6} gives $|\mathcal{L}(\alpha, \beta)|=1$, and  thus $\mathcal{L}_p(\alpha, \beta)$ has no primitive divisor. Therefore by Theorem \ref{thmBH} and Lemmas \ref{lemVO} and \ref{lemBH}, there is no Lehmer number $\mathcal{L}_p(\alpha, \beta)$ except for $p=3,5,7,13$. From Lemma \ref{lemPD}, $(u^2c, -v^2)$ is the parameter corresponds to the pair $(\alpha, \beta)$, so that by Lemma \ref{lemVO}, $p=7,13$ are not possible. This ensures that \eqref{eqn1} has no positive integer solution, except for $p=3,5$. 

We now consider  $p=5$ and in this case  Lemma \ref{lemBH} gives 
\begin{equation}\label{eqq51}
(u^2d,v^2)=(F_{k-2\varepsilon}, 4F_k-F_{k-2\varepsilon}),~~k\geq 3, 
\end{equation}
and 
\begin{equation}\label{eqq52}
(u^2d,v^2)=(L_{k-2\varepsilon}, 4L_k-L_{k-2\varepsilon}), ~~k\ne 1,
\end{equation}
where $\varepsilon=\pm1$. We utilize Lemma \ref{lemfl} in \eqref{eqq51} to get $L_{k+\varepsilon}=v^2$ with $k\geq 3$, which further implies (using Theorem \ref{thmCO}) that $(k,v,\varepsilon)=(4,2,-1)$. This is not possible as $v$ is odd.  
Again utilizing Lemma \ref{lemfl} in \eqref{eqq52}, we get $5F_{k+\varepsilon}=v^2$. As $v=5^mq^n$, so that $F_{k+\varepsilon}=5(5^{m-1}q^n)^2$ and thus by Lemma \ref{lemKK}, one gets $5^{m-1}q^n=1$ which is not possible as $n\geq 1$. 

For $p=3$, we have $v=3^mq^n$ and thus equating the imaginary parts in \eqref{eqp2}, we get 
$$\lambda_1\lambda_2(3u^3d-3^{2m}q^{2n})=4.$$
Since $m\geq 1$, so that the above equation has no integers solution.

$\blacksquare$ $v=p^{m_1}q^n~~(0\leq m_1\leq m-1)$.  In this case, \eqref{eqp6} implies $|\mathcal{L}_p(\alpha, \beta)|=p^{m-m_1}$, and by Lemma \ref{lemPE}, $\mathcal{L}_q(\alpha, \beta)$ has no primitive divisor. In this case too, \eqref{eqp2} has no solution which follows mutatis mutandis as in the last case.   

$\blacksquare$ $v=p^{m_1}q^{n_1}$ with $0\leq m_1\leq m$ and $0\leq n_1\leq n-1$. In the case, $|\mathcal{L}_p(\alpha, \beta)|=p^{m-m_1}q^{n-n_1}$ and by Lemma \ref{lemPE}, $q$ is the only primitive divisor of $\mathcal{L}_p(\alpha, \beta)$. Therefore the previous method is not utilizable. To proceed further, we rewrite \eqref{eqp4} as follows (in this case):
\begin{equation}\label{eqx}
x=\frac{|u\mathcal{R}(d,u,p^{m_1}q^{n_1},p)|}{2^{p-1}}
\end{equation}
and
\begin{equation}\label{eqpq}
p^{m-m_1}q^{n-n_1}=\frac{\lambda_1\lambda_2\mathcal{I}(d,u,p^{m_1}q^{n_1},p)}{2^{p-1}}.
\end{equation}
Employing \eqref{eqi} in \eqref{eqpq}, we have $p^{m-m_1}q^{n-n_1}\equiv \lambda \left(p^{m_1}q^{n_1}\right)^{p-1}\pmod p$, where $\lambda=\pm 1$. As $p$ and $q$ are distinct primes, so that by Fermat's little theorem $p^{m-m_1}q^{n-n_1}\equiv \lambda p^{m_1(p-1)}\pmod p$ which ensures that $m>m_1>0$. Again utilizing, \eqref{eqi} in \eqref{eqpq}, we get $p^{m-m_1}q^{n-n_1}\equiv \dfrac{\lambda pu^{p-1}d^{(p-1)/2}}{2^{p-1}}\pmod{p^2}$. Since $\gcd(p, u)=1$, so that by Fermat's little theorem, we have $p^{m-m_1-1}q^{n-n_1}\equiv \lambda d^{(p-1)/2}\pmod{p}$. This is possible only when $m=m_1+1$ and thus \eqref{eqpq} can be written as 
$$pq^{n-n_1}=\frac{\lambda_1\lambda_2}{2^{p-1}}\left( pu^{p-1}d^{(p-1)/2}+p^{2m-1}q^{2n_1}T \right),$$ for some $T\in\mathbb{Z}$.
As $m\geq 1$, so that it implies 
\begin{equation}\label{eqf}
q^{n-n_1}=\frac{\lambda_1\lambda_2}{2^{p-1}}\left( u^{p-1}d^{(p-1)/2}+p^{2m-2}q^{2n_1}T \right).
\end{equation}
If $n_1\geq 1$, then \eqref{eqf} implies that  
$$2^{p-1}q^{n-n_1}\equiv\lambda_1\lambda_2 u^{p-1}d^{(p-1)/2}\pmod q.$$
Since $\gcd(q, ud)=1$, so that it ensures that $n=n_1$ which contradicts the fact that $n_1\leq n-1$. Therefore $n_1=0$, and hence \eqref{eqf} becomes 
\begin{equation}\label{eqf1}
2^{p-1}q^n=\lambda_1\lambda_2\left( u^{p-1}d^{(p-1)/2}+p^{2m-2}T \right).
\end{equation}
Since $p$ is odd prime with $p\nmid ud$ and $m-1=m_1>0$, so that reading \eqref{eqf1} modulo $p$ gives 
$q^n\equiv \pm 1\pmod p$. 
Thus we can conclude that $\eqref{eqn1}$ has no integer solutions provided $q^n\not\equiv \pm 1\pmod p$. While $q^n\equiv \pm 1\pmod p$, the possible integer solutions are given by \eqref{eqx} and \eqref{eqp3} with $m_1=m-1$ and $n_1=0$. This completes the proof. 
\section{Concluding Remark}
Here, we look into a generalized version of Theorem \ref{thm}. We mainly consider the following Diophantine equation:
\begin{equation}\label{eqc1}
dx^2+p^{2m}q^{2n}=4y^N,~~ ~~x,y,m,n, N\in\mathbb{N},  p\mid N,  \gcd(x,y)=1, \gcd(p, h(-d))=1,   
\end{equation}
where $d$ is a square-free positive integer, and $p$ and $q$ are distinct odd primes.
\begin{thm}\label{thmc} Let $d, p$ and $q$ be as in Theorem \ref{thm}.  Assume that $N$ is a positive integer satisfying $p\mid N$ and $\gcd(N,2h(-d))=1$. Then
\begin{itemize}
\item[(i)] if $d\equiv 1,2\pmod 4$, then \eqref{eqc1} has no integer solution;
\item[(ii)] if $d\equiv 3\pmod 4$, then  \eqref{eqc1} has no positive integer solution when $q^n\not\equiv \pm 1\pmod p$. Moreover when $q^n\equiv \pm 1\pmod p$, \eqref{eqc1} has solutions only if $2^{\frac{N}{p}-1}p^{m-1}= |\mathcal{I}(d,u',1,N/p)|$ with $N/p$ prime and $u'\in\mathbb{N}$, and such solutions (if exist) are given by 
$$(x,y)=\left(\frac{|u\mathcal{R}(d,u,p^{m-1},p)|}{2^{p-1}},  \frac{u'^2d+p^{2\delta(m-1)}}{4}\right),$$
where $u=\begin{cases}
\dfrac{|u'\mathcal{R}(d,u',p^{m-1},N/p)|}{2^{\frac{N}{p}-1}}, \text{ if } \delta=0;\\
u', \text{ if } \delta=1;
\end{cases}$\\
with $u'$ a suitable positive integer.
\end{itemize}
\end{thm}

\begin{proof}[{\bf Proof of Theorem \ref{thmc}}]
Let $(x,y,m,n, N)$ be a positive integer solution of \eqref{eqc1} for a given square integer $d>1$ and a pair of distinct odd primes $p, q$. Then as in the proof of Theorem \ref{thm}, $x$ is odd, $d\equiv 3\pmod 4$ and $\gcd(pq, dx)=1$. 

As $p\mid N$, so that  \eqref{eqc1} can be written as
\begin{equation}\label{eqc2}
dx^2+p^{2m}q^{2n}=4Y^p 
\end{equation}
by putting $Y=y^{N/p}$. By Theorem \ref{thm}, \eqref{eqc2} has no solution in positive integers provided $q^n\not\equiv \pm 1\pmod p$, and hence \eqref{eqc1} too has no solution in positive integers.

We now consider $q^n\equiv \pm 1\pmod p$. Then by Theorem \ref{thm},  the possible solutions of \eqref{eqc2} are given by 
\begin{equation}\label{eqc3}
(x,Y)=\left(\frac{|u\mathcal{R}(d,u,p^{m-1},p)|}{2^{p-1}},  \frac{u^2d+p^{2(m-1)}}{4}\right).
\end{equation}

If $N/p>1$, then we can write $N=pt$ for some odd integer $t\geq 3$ and thus from \eqref{eqc3}, one gets 
\begin{equation}\label{eqc4}
du^2+p^{2(m-1)}=4y^t, ~~u\geq 1, y>1, \gcd(u,y)=1, \gcd(t,2h(-d))=1.
\end{equation}
By \cite[Theorem 2.1]{CHS21}, if $t$ has a prime factor $\ell$ satisfying $2^{\ell-1}p^{m-1}\ne |\mathcal{I}(d,u',1,\ell)|$ for some suitable positive integer $u'$, then \eqref{eqc4} has no solution $(d, u, y, m, t)$ and thus \eqref{eqc1} too has no integer solution. Further, \eqref{eqc4} has  solutions only if $t$ is prime satisfy $2^{t-1}p^{m-1}= |\mathcal{I}(d,u',1,t)|$ and those are given by
$$(u, y)=\left(\frac{u'|\mathcal{R}(d,u',1,t)|}{2^{t-1}}, \frac{u'^2d+1}{4}\right).$$
Thus the corresponding solutions of \eqref{eqc1} are given by 
$$(x,y)=\left(\frac{|u\mathcal{R}(d,u,p^{m-1},p)|}{2^{p-1}}, \frac{u'^2d+1}{4}\right),$$
where $u=\dfrac{u'|\mathcal{R}(d,u',1,N/p)|}{2^{\frac{N}{p}-1}}$ and $u'\geq 1$ is a suitable integer. 

Finally for $N/p=1$, one can conclude that the solutions of \eqref{eqc1} are as follows:
$$(x,y)=\left(\frac{|u\mathcal{R}(d,u,p^{m-1},p)|}{2^{p-1}},  \frac{u^2d+p^{2(m-1)}}{4}\right),$$
where $u\geq 1$ is an odd integer.
 This completes the proof. 
\end{proof}

\section*{acknowledgements}
This work is supported by the grants SERB MATRICS Project No. MTR/2017/001006 and SERB-NPDF (PDF/2017/001958), Govt. of India. The authors are grateful to the anonymous referee for careful reading and valuable comments which have helped to improve this paper. 

\end{document}